\documentclass[]{article}
\pdfoutput=1
\usepackage[T1]{fontenc}
\usepackage{amsmath}
\usepackage[english]{babel}
\usepackage{amsthm}
\usepackage{mathtools}
\usepackage{amsfonts}
\usepackage{amssymb}
\usepackage{stmaryrd}
\usepackage{enumitem}
\usepackage{mathrsfs}
\usepackage{tikz}
\usetikzlibrary{cd}
\usetikzlibrary{calc}
\usetikzlibrary{positioning}
\usetikzlibrary{scopes}
\usetikzlibrary{decorations.pathmorphing}
\usepackage[most]{tcolorbox}
\usepackage{varioref}
\usepackage{array}
\usepackage{makecell}
\usepackage{booktabs}
\usepackage[raggedright]{titlesec}

\usepackage{hyperref}
\usepackage[initials, nobysame, backrefs, msc-links]{amsrefs}
\hypersetup{
  colorlinks, 
  linkcolor=-red!75!green!50, 
  citecolor=-red!75!green!50,
  urlcolor=green!30!black
}

\newtheorem{Lemma}{Lemma}
\newtheorem{Proposition}[Lemma]{Proposition}
\newtheorem{Corollary}[Lemma]{Corollary}

\newtheorem{TheoremIntro}{Theorem}

\theoremstyle{definition}

\theoremstyle{remark}
\newtheorem{Remark}[Lemma]{Remark}

\newskip\originalparindent
\AtBeginDocument{\originalparindent\parindent}
\newtcolorbox{note}{
        breakable, 
        enhanced, 
        colback=red!5!white, 
        colframe=red!75!black,
        before upper={\parindent\originalparindent}
        }

\newcommand\NN{\mathbb{N}}
\newcommand\kk{\Bbbk}
\renewcommand\O{\mathcal{O}}
\DeclareMathOperator\soc{soc}
\let\top\relax
\DeclareMathOperator\top{top}
\DeclareMathOperator\rad{rad}

\DeclareMathOperator\GL{GL}
\DeclareMathOperator\Rep{Rep}
\newcommand\newterm[1]{\textbf{\itshape\color{blue!50!black}#1}}

\DeclareMathOperator\Ext{Ext}
\let\hom\relax
\DeclareMathOperator\hom{Hom}

\newcommand\A{\mathcal{A}}

\usepackage{caption}
\newlist{thmlist}{enumerate}{1}
\setlist[thmlist]{
        nolistsep,
        ref={\mdseries\textup{(\emph{\roman*})}},
        label={\mdseries\textup{(\emph{\roman*})}},
        }
\newcommand\thmitem[1]{\textup{(\emph{\romannumeral #1})}}

\newlist{tfae}{enumerate}{1}
\setlist[tfae]{
        nolistsep,
        ref={\mdseries\textup{(\emph{\alph*})}},
        label={\mdseries\textup{(\emph{\alph*})}},
        }
\makeatletter
\newcommand\tfaeitem[1]{{\textup{(\emph{\@alph #1})}}}

\makeatother

\setlist[itemize,2]{label=$\triangleright$}

\makeatletter
\newlength{\@thlabel@width}
\newcommand{\fixhspace}{%
        \settowidth{\@thlabel@width}{\hskip0.5em\itshape1.}%
        \sbox{\@labels}{%
                \unhbox\@labels%
                \hspace{\dimexpr-\leftmargin+\labelsep+\@thlabel@width-\itemindent}
                }
        }
\makeatother

\newcommand\claim[2][.8]{%
  \begin{minipage}{#1\displaywidth}%
  \itshape
  #2
  \end{minipage}%
}

\newcommand\N{\mathcal{N}}

\DeclarePairedDelimiter\abs{\lvert}{\rvert}
\newcommand\rk[1]{\abs{#1}}
\tikzcdset{
  inline/.style={column sep=1em}
}
\tikzset{
  quiver vertex/.style={radius=2pt},
  quiver arrow/.style={
        thick, shorten <=4pt, shorten >=4pt, -{Stealth[#1]},
        },
  small quiver arrow/.style={
        thick, shorten <=2pt, shorten >=2pt, -{Stealth[length=4pt, width=4pt]},
        },
  snake arrow/.style={
    -{Stealth[#1]},
    decorate, thick, shorten <=4pt, shorten >=4pt,
    decoration={snake, amplitude=0.5mm, segment length=1mm, post length=5mm,
                pre length=5mm}},
}

\title{A simple homological characterization of string algebras of
finite representation type}

\author{Mariano Suárez-Álvarez}
\date{May 6, 2021}

\begin{document}

\maketitle

\begin{abstract}
We prove that among the finite dimensional algebras of finite
representation type those that are string algebras are precisely the
ones that have the property that the middle term of an arbitrary
extension of indecomposable modules has at most two direct factors. On
the other hand, we show that non-domestic string algebras are very far
from having that property.
\end{abstract}

\vspace{2pc}

\hspace*{\stretch{1}}
\textit{In memory of Andrzej Skowro\'nski}

\vspace{1pc}

Throughout this paper $\kk$ denotes an algebraically closed field,
algebras are finite dimensional $\kk$-algebras and modules over them
finitely generated.

\medskip

It is well-known that the middle term of every almost split short
exact sequence in the category of modules over a string algebra has at
most two direct factors. In this note I~will show that, in fact, this
is true of \emph{all} extensions of indecomposable modules provided
the algebra has finite representation type and, moreover, that this
class of algebras is completely characterized by this property.
Precisely, my main result is:

\begin{TheoremIntro}\label{thm:main}
A finite dimensional algebra of finite representation type is a string
algebra if and only if the middle term of \emph{every} extension
between its indecomposable modules has at most two direct factors.
\end{TheoremIntro}

The proof of the necessity of the condition that I can give is based
on the proof of Theorem~1 in Christine Riedtmann's
paper~\cite{Riedtmann} and involves the theory of degenerations of
modules. The key observation is the following statement that appears
as Proposition~\ref{prop:deg} in the body of this paper:

\begin{TheoremIntro}\label{thm:deg:intro}
Let $A$ be a string algebra of finite representation type and let~$M$
and~$N$ be two $A$-modules with the same dimension vectors. If $M$
degenerates to~$N$, then $N$ has at least as many direct factors
as~$M$.
\end{TheoremIntro}

\noindent The proof of this depends crucially both on the finiteness
of the representation type and on the fact that the middle terms of
almost split short exact sequences of modules over string algebras
have at most two direct factors. 

\medskip

It is natural to wonder what happens to Theorem~\ref{thm:main} if the
hypothesis on the representation type is dropped. What I can prove in
that direction is that some control on the representation type is
needed of one wants a bound on the number of direct factors of an
extension of indecomposable modules: for string algebras we need to be
as close to having finite representation type as possible.

\begin{TheoremIntro}\label{thm:dom:intro}
If a string algebra has non-domestic representation type, then there
are extensions of indecomposable modules with arbitrarily many direct
factors.
\end{TheoremIntro}

Here, that the string algebra, which certainly has tame representation
type, have non-domestic representation type means that the minimal
number~$\mu(d)$ of $1$-parametric families of indecomposable modules
needed to describe the indecomposable modules of each dimension~$d$
is not bounded by a scalar independent of~$d$. Proving
Theorem~\ref{thm:dom:intro} seems to require delving into the
combinatorics of bands. I do this in Section~\ref{sect:bands}, with a
beautiful lemma of Claus Ringel from~\cite{R:generic} as starting
point. 

\medskip

Theorems~\ref{thm:main} and~\ref{thm:dom:intro} leave open the
question of whether one can find a bound for the number of direct
factors in an extension of indecomposable modules over domestic string
algebras or not. This question seems also to require dealing with the
combinatorics of bands, and I expect to treat it elsewhere.

I heartfully thank Claudia Chaio, Cristian Chaparro, María Julia
Redondo,  and Pamela Suárez for many enlightening conversations on
subjects related to those treated here that were the original
motivation for this work.

\bigskip

\noindent\textit{Conventions.} As mentioned above, throughout this
paper~$\kk$ denotes an algebraically closed field. Our algebras are
finite dimensional $\kk$-algebras except possibly for path algebras,
and modules are finite dimensional right modules. If~$\alpha\beta$ is
a path in a quiver, then the arrow~$\alpha$ ends where the
arrow~$\beta$ starts. If $M$ is a module, we denote by~$\abs{M}$ the
number of summands in any decomposition of~$M$ as a direct sum of
indecomposable submodules.

Except when defining what a string algebra is we will implicitly
suppose that all algebras are basic. This is no restriction, as our
results are all Morita invariant, and allows us to assume that
algebras are given by quivers and relations.

\section{Middle terms with at most two direct factors}
\label{sect:sufficiency}

Let $Q$ be a finite quiver, let $\kk Q$ be the corresponding path
algebra, and let $I$ be an admissible ideal in~$\kk Q$. We say,
following Andrei Skowro\'nski and Josef Waschbüsch in~\cite{SW}, that
the presentation $(Q,I)$ is \newterm{special biserial} if the
following two conditions are satisfied.
\begin{enumerate}[label=(S\arabic*), ref=(S\arabic*), left=2em]

\item\label{ax:S1} Every vertex in~$Q$ has in-degree at most~$2$ and
out-degree at most~$2$.

\item\label{ax:S2} If $\alpha$ is an arrow in~$Q$, then there is at
most one arrow~$\beta$ such that $\alpha\beta$ is a path in~$Q$ and
does not belong to~$I$, and there is at most one arrow~$\gamma$ such
that $\gamma\alpha$ is a path in~$Q$ and does not belong to~$I$.

\end{enumerate}
If additionally the following third condition is satisfied we say that
the presentation~$(Q,I)$ is \newterm{string}.
\begin{enumerate}[resume*]

\item\label{ax:S3} The ideal~$I$ is generated by the paths it
contains.

\end{enumerate}
A finite dimensional algebra~$A$ is a \newterm{string algebra} if it
is Morita equivalent to an algebra of the form $\kk Q/I$ for some
finite quiver~$Q$ and some admissible ideal~$I$ in the path
algebra~$\kk Q$ such that the presentation $(Q,I)$ is string. 

\bigskip

The first step on proving our main result is the following.

\begin{Proposition}\label{prop:one}
Let $Q$ be a finite quiver, let $I$ be an admissible ideal in the path
algebra~$\kk Q$ and let $A$ be the quotient algebra~$\kk Q/I$. 
\begin{thmlist}

\item If every extension between indecomposable $A$-modules has at
most two direct factors, then the presentation $(Q,I)$ satisfies the
condition~\ref{ax:S1}.

\item Moreover, if the presentation $(Q,I)$ also satisfies the
condition~\ref{ax:S2}, then it is in fact string.

\end{thmlist}
\end{Proposition}

\begin{proof}
Let us suppose that every extension of indecomposable $A$-modules has
at most two direct factors and, to reach a contradiction, that there
are in~$Q$ three arrows~$\alpha$,~$\beta$ and~$\gamma$ with the same
target vertex~$i$. We let $j_1$,~$j_2$ and~$j_3$ be the source
vertices of those three arrows; notice that the cardinal of the set
$\{i,j_1,j_2,j_3\}$ can be any integer from~$1$ to~$4$. In any case,
$Q$ contains three arrows that look like
  
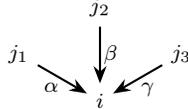
\[
  \begin{tikzpicture}[font=\footnotesize]
  \node[inner sep=0pt, outer sep=0pt] (i) at (0,0) {$i$};
  \node[inner sep=0pt, outer sep=0pt] (j1) at (150:1.25) {$j_1$};
  \node[inner sep=0pt, outer sep=0pt] (j2) at (90:1.25) {$j_2$};
  \node[inner sep=0pt, outer sep=0pt] (j3) at (30:1.25) {$j_3$};
  \foreach \a/\b/\where/\x in {
        j1/i/north east/\alpha, 
        j2/i/west/\beta, 
        j3/i/north west/\gamma}
    \draw[quiver arrow] (\a) -- node[inner sep=1pt, anchor=\where] {$\x$} (\b);
  \end{tikzpicture}
  \]
and we can construct an extension of a $5$-dimensional indecomposable
module by the simple module at~$i$ that, if we draw dimension vectors,
looks as follows:
  \[
  \def\enlarge{
      \useasboundingbox
        ([shift={(3mm,0mm)}]current bounding box.north east) 
        rectangle 
        ([shift={(-3mm,0mm)}]current bounding box.south west);
  }
  \def\d{0.6}
  \begin{tikzcd}[column sep=2em]
  0 \arrow[r]
    & 
      \begin{tikzpicture}[font=\footnotesize, scale=0.5,
        baseline={($(current bounding box.center)+(0,-1mm)$)}]
      \node[inner sep=0pt, outer sep=0pt] (i) at (0,0) {$1$};
      \node[inner sep=0pt, outer sep=0pt] (j1) at (150:\d) {$0$};
      \node[inner sep=0pt, outer sep=0pt] (j2) at (90:\d) {$0$};
      \node[inner sep=0pt, outer sep=0pt] (j3) at (30:\d) {$0$};
      \enlarge
      \end{tikzpicture}
      \arrow[r]
    & 
      \begin{tikzpicture}[font=\footnotesize, scale=0.5,
        baseline={($(current bounding box.center)+(0,-1mm)$)}]
      \node[inner sep=0pt, outer sep=0pt] (i) at (0,0) {$1$};
      \node[inner sep=0pt, outer sep=0pt] (j1) at (150:\d) {$1$};
      \node[inner sep=0pt, outer sep=0pt] (j2) at (90:\d) {$0$};
      \node[inner sep=0pt, outer sep=0pt] (j3) at (30:\d) {$0$};
      \enlarge
      \end{tikzpicture}
      \oplus
      \begin{tikzpicture}[font=\footnotesize, scale=0.5,
        baseline={($(current bounding box.center)+(0,-1mm)$)}]
      \node[inner sep=0pt, outer sep=0pt] (i) at (0,0) {$1$};
      \node[inner sep=0pt, outer sep=0pt] (j1) at (150:\d) {$0$};
      \node[inner sep=0pt, outer sep=0pt] (j2) at (90:\d) {$1$};
      \node[inner sep=0pt, outer sep=0pt] (j3) at (30:\d) {$0$};
      \enlarge
      \end{tikzpicture}
      \oplus
      \begin{tikzpicture}[font=\footnotesize, scale=0.5,
        baseline={($(current bounding box.center)+(0,-1mm)$)}]
      \node[inner sep=0pt, outer sep=0pt] (i) at (0,0) {$1$};
      \node[inner sep=0pt, outer sep=0pt] (j1) at (150:\d) {$0$};
      \node[inner sep=0pt, outer sep=0pt] (j2) at (90:\d) {$0$};
      \node[inner sep=0pt, outer sep=0pt] (j3) at (30:\d) {$1$};
      \enlarge
      \end{tikzpicture}
      \arrow[r]
    & \begin{tikzpicture}[font=\footnotesize, scale=0.5,
        baseline={($(current bounding box.center)+(0,-1mm)$)}]
      \node[inner sep=0pt, outer sep=0pt] (i) at (0,0) {$2$};
      \node[inner sep=0pt, outer sep=0pt] (j1) at (150:\d) {$1$};
      \node[inner sep=0pt, outer sep=0pt] (j2) at (90:\d) {$1$};
      \node[inner sep=0pt, outer sep=0pt] (j3) at (30:\d) {$1$};
      \enlarge
      \end{tikzpicture}
      \arrow[r]
    & 0
  \end{tikzcd}
  \]
That there are such $A$-modules is immediate: we certainly have $\kk
Q$-modules with these descriptions and every path of length~$2$ acts
as zero on them, so they are in fact $A$-modules. Since the existence
of this short exact sequence contradicts the hypothesis, we see that
every vertex in~$Q$ has in-degree at most~$2$. A dual argument also
bounds the out-degree of the vertices of~$Q$, so $(Q,I)$ satisfies the
condition~\ref{ax:S1} in the definition of string presentations. This
proves part~\thmitem{1} of the proposition.

In order to prove part~\thmitem{2} let us now assume the hypothesis
that the presentation~$(Q,I)$ satisfies the condition~\ref{ax:S2}, so
that it is in fact special biserial. According to the corollary to
Lemma~1.1 in~\cite{SW}, the ideal $I$ is generated by paths and
binomials, that it, linear combinations of \emph{two} parallel paths.
It follows from this that in order to show that the presentation
$(Q,I)$ is string it is enough that we prove that whenever a linear
combination of two parallel paths is in~$I$ then both paths also
belong to~$I$.

Let us then suppose that there is in~$I$ an element of the
form~$u-\lambda v$ with $u$ and $v$ different parallel paths that do
not belong to~$I$ and $\lambda$ a non-zero scalar. Multiplying by a
$\lambda^{-1}$ if needed we may assume that moreover $u$ is not longer
that~$v$. Both~$u$ and~$v$ have length at least~$2$ because the
ideal~$I$ is admissible.

Suppose for a moment that both~$u$ and~$v$ start with the same arrow.
Since neither~$u$ nor~$v$ is in~$I$ and $u$ is not longer that $v$,
the condition~\ref{ax:S2} implies that there is a path~$w$, which is
necessarily closed, such that $v=uw$. Since $u\neq v$ this path~$w$
has positive length and, because the ideal~$I$ is admissible, there is
a positive integer~$k$ such that $w^k\in I$. Now we have that
$u(1-\lambda w)=u-\lambda v\in I$, so that
  \[
  I \ni u(1-\lambda w)\sum_{t=0}^{k-1}\lambda^tw^t
    = u - \lambda^kw^k,
  \]
and this is absurd since~$u\not\in I$. A symmetric argument shows, of
course, that the paths $u$ and~$v$ cannot end in with the same arrow.

It follows from all this and the fact that both $u$ and $v$ have
length at least~$2$ that there are arrows~$\alpha$,~$\beta$,~$\gamma$
and~$\delta$ in~$Q$ and paths $\bar u$ and~$\bar v$, possibly of
length~$0$, such that $\alpha\neq\beta$, $\gamma\neq\delta$,
$u=\alpha\bar u\gamma$ and $v=\beta\bar v\delta$. 

Let us write $i$ and~$j$ for the common source and target of~$u$
and~$v$, respectively. If $\eta$ is an arrow coming out of~$j$, then
because $\gamma\neq\delta$ we have that one of~$u\eta$ or~$v\eta$ is
in~$I$ and therefore, since $u-\lambda v\in I$, that so is the other.
Similarly, we have that $\epsilon u$ and $\epsilon v$ are in~$I$ for
all arrows~$\epsilon$ which have~$i$ as target.

Let now $w$ be a path of positive length that starts from~$i$, is not
in~$I$, and such that there is no arrow~$\eta$ such that $w\eta\not\in
I$. The arrows coming out of~$i$ are~$\alpha$ and~$\beta$, so one of
them is the first arrow in~$w$. Let us suppose, for example, that
$\alpha$ is. In view of~\ref{ax:S2}, one of~$u$ or~$w$ is a prefix of
the other, and since for all arrows~$\eta$ we have $u\eta\in I$ and
$w\eta\in I$, we see that in fact $w=v$. Similarly, if $\beta$ were
the first arrow in~$w$ we would have that $w=v$. We thus see that $u$
and $v$ are the only paths starting from~$i$, not in~$I$, and which
cannot be extended on the right: in other words, they span the socle
of the indecomposable projective module~$e_iA$, which is the
projective cover of the simple module~$S_i\coloneqq e_iA/\rad e_iA$.
Moreover, since $u-\lambda v\in I$, we see that that socle is of
dimension~$1$, so isomorphic to the simple module~$S_j$. A dual
argument shows that $u$ and~$v$ are the only paths not in~$I$ that end
in~$j$ and that cannot be extended on the left: the top of the
injective envelope of the simple~$S_j$ is thus $S_i$. It follows from
this that the indecomposable projective module $e_iA$ is in fact also
injective\footnote{The injective envelope $\soc e_iA\cong
S_j\hookrightarrow D(Ae_j)$ factors through a map $f:e_iA\to D(Ae_j)$,
and this map is injective because it does not vanish on the simple
socle of~$e_iA$. Similarly, the projective cover $e_iA\to S_i\cong\top
D(Ae_j)$ factors through a map $e_iA\to D(Ae_j)$, and this map is
surjective because its composition with $D(Ae_j)\to\top D(Ae_j)$ is
surjective: it follows from this that the map~$f$ is an isomorphism.}
and we therefore have an
almost split sequence
  \[
  \begin{tikzcd}
  0 \arrow[r]
    & \rad e_iA \arrow[r]
    & e_iA\oplus\dfrac{\rad e_iA}{\soc e_iA} \arrow[r]
    & \dfrac{e_iA}{\soc e_iA} \arrow[r]
    & 0
  \end{tikzcd}
  \]
The two ends of this extension are indecomposable, because the socle
and the top of~$e_iA$ are simple, so the middle term has at most two
direct factors: it follows that the quotient $\rad e_iA/\soc e_iA$ is
zero or indecomposable. This is absurd, since in our situation we have
  \[
  \dfrac{\rad e_iA}{\soc e_iA} 
  \cong \dfrac{\alpha A}{uA} \oplus \dfrac{\beta A}{vA}
  \]
and both summands are non-zero.
\end{proof}

We have written Proposition~\ref{prop:one} in the slightly weird form
that we have so as to make it explicit that it is
condition~\ref{ax:S2} that is difficult to satisfy. One way to do that
is to restrict ourselves to algebras of finite representation type:

\begin{Proposition}
Let $Q$ be a finite quiver, let $I$ be an admissible ideal in the path
algebra~$\kk A$, and let $A\coloneqq\kk Q/I$. If $A$ has finite
representation type and every extension between indecomposable
$A$-modules has at most two direct factors, then $(Q,I)$ is a string
presentation.
\end{Proposition}

This gives us half of the Theorem~\ref{thm:main} stated in the
introduction.

\begin{proof}
Suppose that $A$ satisfies both hypotheses. In view of
Proposition~\ref{prop:one}, the presentation~$(Q,I)$ satisfies the
condition~\ref{ax:S1} and to prove that it is string it is enough to
that it also satisfies the condition~\ref{ax:S2}. Since every
extension between indecomposable $A$-modules has at most two direct
factors, we have in particular that the middle term of every almost
split exact sequence of $A$-modules has at most two direct factors
and, since $A$ has finite representation type, Theorem~4.6
in~\cite{AR:uniserial} tells us that the radical of every
non-uniserial indecomposable projective module is the \emph{direct}
sum of two uniserial modules. 

Suppose there are arrows~$\alpha$, $\beta$ and~$\gamma$ in~$Q$ such
that $\alpha$ ends where~$\beta$ and~$\gamma$ begin,
$\beta\neq\gamma$, and the paths $\alpha\beta$ and~$\alpha\gamma$ are
both not in~$I$, and let $i$ the source of~$\alpha$. As $A$ has finite
representation type, the quiver~$Q$ does not have parallel arrows, and
$\beta$ and~$\gamma$ end at different vertices. Moreover, the
indecomposable projective module $e_iA$ is not uniserial: if it were,
we would have $\dim\rad^ne_iP/\rad^{n+1}e_iP\leq1$ for all $n\in\NN_0$
and what we have tells us that the classes of~$\alpha\beta$
and~$\alpha\gamma$ are linearly independent elements
in~$\rad^2e_iA/\rad^3e_iA$. The right ideal $\alpha A$ is
indecomposable because it has a simple top, is not uniserial, and is
contained in~$\rad e_iA$: this is absurd, as we have seen that this
radical is a direct sum of two uniserial modules.
\end{proof}

\section{Degenerations and the number of direct factors of a module}
\label{sect:deg}

Let us recall from Bongartz's paper~\cite{Bongartz} the notion of
degeneration of modules, deferring to that paper and to
Riedtmann's~\cite{Riedtmann} for anything else on the subject.

Let $A$ be a finite-dimensional algebra over our algebraically closed
field~$\kk$. If~$V$ is a finite-dimensional vector space, then we
consider the set
  \[
  L_A(V)\coloneqq\hom(V\otimes A,V)
  \]
of all linear maps $\rho:V\otimes A\to V$ and its subset~$\Rep_A(V)$
consisting of those maps~$\rho$ that turn $V$ into an $A$-module
$V_\rho$. The set~$L_A(V)$ is an affine space over~$\kk$, and it is
easy to check that $\Rep_A(V)$ is an affine algebraic variety in it.
The group~$\GL(V)$ of linear automorphisms $V\to V$ acts naturally
on~$L_A(V)$ by conjugation and that action restricts to one
on~$\Rep_A(V)$. Two points $\rho$ and~$\rho'$ of~$\Rep_A(V)$ are in
the same $\GL(V)$-orbit if and only if there is an isomorphism
$V_\rho\cong V_{\rho'}$ of $A$-modules. In particular, a
finite-dimensional $A$-module~$M$ such that $\dim M=\dim V$ determines
a unique orbit $\O(M)$ in~$\Rep_A(V)$, the one consisting of the
points~$\rho$ such that there is an isomorphism of $A$-modules $M\cong
V_\rho$.

If $M$ and $N$ are two $A$-modules with $\dim M=\dim N=\dim V$, then
we say that $M$ \newterm{degenerates} to~$N$ if the orbit~$\O(N)$ is
contained in the Zariski closure of the orbit~$\O(M)$. 

\bigskip

If a module~$M$ degenerates to another module~$M$, then $M$ and~$N$
have the same dimension vectors, that is, they have the same
composition factors. On the other hand, there is in general very
little that one can say about the relation between the numbers of
their direct factors~$\rk{M}$ and~$\rk{N}$. A natural guess is that
the inequality $\rk{M}\leq\rk{N}$ should hold, but it does not:
Bongartz gives in~\cite{Bongartz}*{Example 7.1} examples that show
that over the self-injective $4$-dimensional algebra
$\kk[X,Y]/(X^2,XY,Y^2)$ modules with arbitrarily many direct summands
degenerate to indecomposable ones. For string algebras of finite
representation type a little miracle occurs, though, and the natural
guess is on the mark.

\begin{Proposition}\label{prop:deg}
Let $A$ be a string algebra of finite representation type and let $M$
and $N$ be two $A$-modules with the same dimension vector. If $M$
degenerates to~$N$, then $\rk{M}\leq\rk{N}$.
\end{Proposition}

This is the Theorem~\ref{thm:deg:intro} stated in the introduction.

\begin{proof}
Let us suppose that $M$ degenerates to~$N$ and for each $A$-module $V$
let us write 
  \[
  \delta(V) \coloneqq \dim\hom_A(V,N)-\dim\hom_A(V,M).
  \]
According to \cite{Riedtmann}*{Proposition 2.1} we have
$\delta(V)\geq0$ for all $A$-modules~$V$ and, since $M$ and $N$ have
the same dimension vector, we have $\delta(V)=0$ if $V$ is projective
and indecomposable. For each non-projective indecomposable $A$-module
$V$ let $\Sigma_V$ be the Auslander--Reiten sequence starting at~$V$,
  \[
  \begin{tikzcd}
  0 \arrow[r]
    & \tau V \arrow[r]
    & E_V \arrow[r]
    & V \arrow[r]
    & 0
  \end{tikzcd}
  \]
and let us consider the short exact sequence
  \[
  \begin{tikzcd}
  0 \arrow[r]
    & X \arrow[r]
    & Y \arrow[r]
    & Z \arrow[r]
    & 0
  \end{tikzcd}
  \]
that is the direct sum $\bigoplus_{V}\Sigma_V^{\delta(V)}$ of the
short exact sequences~$\Sigma_V$ for all non-projective indecomposable
modules~$V$, each taken with multiplicity~$\delta(V)$. This direct sum
is finite, of course, because the algebra~$A$ has finite
representation type. In the proof of \cite{Riedtmann}*{Theorem 1.1} it
is established that we have 
  \[
  M\oplus X\oplus Z \cong N\oplus Y
  \]
and it follows from this and the Krull--Remak--Schmidt theorem that
  \begin{align*}
  \rk{N} - \rk{M} 
        =  \rk{X} + \rk{Z} - \rk{Y}
       &= \sum_V\delta(V)\bigl(\rk{\tau V}+\rk{V}-\rk{E_V}\bigr) \\
       &= \sum_V\delta(V)\bigl(2-\rk{E_V}\bigr).
  \end{align*}
Since we know that $\rk{E_V}\leq2$ for all~$V$ ---for example, from
the explicit construction done by Butler and Ringel in~\cite{BR} of
all the Auslander--Reiten sequences for~$A$--- we see at once that
$\rk{N}-\rk{M}\geq0$, which is precisely what the proposition claims.
\end{proof}

It may be helpful to the reader of~\cite{Riedtmann} to know that the
unpublished result of Auslander of which Riedtmann makes use ---that
two modules $M$ and~$N$ over an arbitrarily finite dimensional
algebra~$\Lambda$ are isomorphic as soon as we have
$\dim\hom_\Lambda(U,M)=\dim\hom_\Lambda(U,N)$ for all indecomposable
$\Lambda$-modules $U$--- follows easily from Lemma 1.2 in Bongartz's
paper~\cite{Bongartz}.

\bigskip

Let us state a useful special case of the proposition:

\begin{Corollary}\label{coro:two}
Let $A$ be a string algebra of finite representation type. If
  \[
  \begin{tikzcd}[column sep=1.25em]
  0 \arrow[r]
    & N \arrow[r]
    & E \arrow[r]
    & M \arrow[r]
    & 0
  \end{tikzcd}
  \]
is a short exact sequence of $A$-modules, then $\rk{E}\leq\rk{M}+\rk{N}$.
\end{Corollary}

\begin{proof}
It follows from Lemma~1.1 in~\cite{Bongartz} or the lemma that appears
in the proof of Corollary~2.3 in~\cite{Riedtmann} that in the
situation of the lemma the module~$E$ degenerates to~$M\oplus N$, so
the corollary is a direct consequence of the proposition.
\end{proof}

This corollary gives us the half of Theorem~\ref{thm:main} that we did
not prove in the previous section.

\begin{proof}[Proof of Theorem~\ref{thm:main}]
Proposition~\ref{prop:one} tells us that an algebra that satisfies the
condition in the statement of the theorem is a string algebra.
Conversely, Corollary~\ref{coro:two} above immediately implies us that
a string algebra of finite representation type satisfies that
condition.
\end{proof}

Th.\,Brüstle, G.\,Douville, K.\,Mousand, H.\,Thomas and E.\,Yıldırım
in~\cite{BDMTY} and \.{I}.\,\c{C}anak\c{c}\i, D. Pauksztello and
S.~\,Schroll in~\cite{CPS} have constructed bases for the extension
spaces~$\Ext^1_A(M,N)$ corresponding to all choices of indecomposable
modules~$M$ and~$N$ over a \emph{gentle} algebra $A$ of finite
representation type, and the middle terms of the extensions
representing the elements of those bases all have at most two direct
factors --- as they must, according to Corollary~\ref{coro:two} above,
since gentle algebras are string algebras. It should be noted that it
does not follow from this that \emph{all} extensions between
indecomposable modules have middle terms with at most two factors,
though. For example, if we write modules over the path algebra~$\kk Q$
of the quiver~$Q$
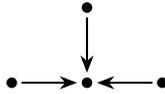
  \[
  \begin{tikzpicture}[thick, -{Stealth[]}, scale=1]
  \foreach \i/\p in {1/{(0,0)}, 2/{(-1,0)}, 3/{(0,1)}, 4/{(1,0)}}
    \fill \p circle(2pt) node[name=\i] {};
  \draw (2) -- (1);
  \draw (3) -- (1);
  \draw (4) -- (1);
  \end{tikzpicture}
  \]
using dimension vectors, it turns out that
the vector space 
  \[
  V\coloneqq\Ext^1_{\kk Q}(
  \begin{tikzpicture}[xscale=0.2, yscale=0.25, 
        font=\footnotesize, baseline={(0,0)}]
  \node at (0,0) {2};
  \node at (-1,0) {1};
  \node at (0,1) {1};
  \node at (+1,0) {1};
  \end{tikzpicture}
  ,
  \begin{tikzpicture}[xscale=0.2, yscale=0.25, 
        font=\footnotesize, baseline={(0,0)}]
  \node at (0,0) {1};
  \node at (-1,0) {0};
  \node at (0,1) {0};
  \node at (+1,0) {0};
  \end{tikzpicture}
  )
  \]
has dimension~$2$, that there are in it three $1$-dimensional
subspaces whose elements correspond to extensions with two direct
factors in the middle, and that all the other elements of~$V$
correspond to extensions with three direct factors in the middle. This
example and many others lead to considerations of geometric nature
and, eventually, to the degeneration argument that we have used.

\begin{Remark}
The proof of Proposition~\ref{prop:deg} shows that over an algebra of
finite representation type if a module~$M$ degenerates to another
module~$N$, then $\rk{M}-\rk{N}$ can be estimated from the knowledge
of the number of middle terms of Auslander--Reiten short exact
sequences and the function~$\delta$. It would be interesting to know
if this observation can be put to use.
\end{Remark}

\section{String algebras of infinite representation type}
\label{sect:bands}

In this section we will study string algebras of infinite
representation type and to be able to talk about them we start by
recalling as little as possible of the language of bands --- for
convenience we will use the version presented in~\cite{R:compact} and
\cite{R:generic}, because our constructions will depend on results
presented there using that version of the language.

Let $(Q,I)$ be a string presentation. Let us denote by $Q_0$ and $Q_1$
the sets of vertices and arrows of~$Q$, respectively, and by
$s$,~$t:Q_1\to Q_0$ the functions mapping each arrow to its source and
its target, respectively.

Let $\hat Q$ be the quiver obtained from~$Q$ by adding for each
arrow~$\alpha$ in~$Q_1$ a new arrow~$\alpha^{-1}$ going in the
opposite direction, so that $s(\alpha^{-1})=t(\alpha)$ and
$t(\alpha^{-1})=s(\alpha)$. We call the arrows of~$\hat Q$
\newterm{letters}, and a letter is \newterm{direct} or
\newterm{inverse} according to whether it is in~$Q_1$ or not. There is
an involution $l\in\hat Q_1\mapsto l^{-1}\in\hat Q_1$ mapping each
direct letter~$\alpha$ to its formal inverse~$\alpha^{-1}$, which we
call \newterm{inversion}.

A \newterm{walk} in~$Q$ is a path in~$\hat Q$, possibly one of those
of length zero that correspond to the vertices of~$Q$. The
\newterm{inverse} of a walk $w=l_1\cdots l_r$ of positive length is
the walk $w^{-1}\coloneqq l_r^{-1}\cdots l_1^{-1}$, while the each
walk of length zero is its own inverse. A walk is \newterm{direct} or
\newterm{inverse} if all its letters are direct or inverse,
respectively, and it is \newterm{serial} if it is either direct or
inverse. There is an obvious extension of the source and target
functions~$s$ and~$t$ to the set of all walks in~$Q$, and if $u$ and
$v$ are walks such that $t(u)=s(v)$ we can construct a new walk $uv$
simply by concatenation; this partially defined operation on the set
of walks is associative in the appropriate sense. A walk~$u$ is a
\newterm{factor} of a walk~$w$ if there are walks~$v_1$ and~$v_2$ such
that the product $v_1uv_2$ is defined and equal to~$w$, and a
\newterm{prefix} if we can choose~$v_1$ of length zero. In this last
case we write $u\sqsubseteq w$.

A walk~$w$ is a \newterm{word} if \emph{(i)} $w$ has no factor of the
form $ll^{-1}$ with $l$ a letter and \emph{(ii)} neither $w$
nor~$w^{-1}$ has a factor that belongs to~$I$. A word~$w$ is
\newterm{cyclic} if it is not serial and $w^2$ is also a word, and a
cyclic word is \newterm{primitive} if it is not of the form~$v^r$ for
some cyclic word~$v$ and some integer~$r\geq2$.

Let $A\coloneqq\kk Q/I$ be the algebra presented by~$(Q,I)$. From a
word~$w$ we can construct a $A$-module $M(w)$, called the
\newterm{string module} corresponding to~$w$, and from a primitive
cyclic word~$w$, a positive integer~$n$ and a non-zero
scalar~$\lambda\in\kk\setminus0$ we can construct an $A$-module
$B(w,\lambda,n)$, called the \newterm{band module} corresponding to
those parameters. A well-known result ---due essentially to Gel\cprime
fand and Ponomarev \cite{GP}--- states that in this way we obtain, up
to isomorphism, all the indecomposable $A$-modules, and each of them
once except for easily controlled repetitions: if $w$ is a word, then
the string modules obtained from~$w$ and from~$w^{-1}$ are
isomorphic, and if $w$ is a primitive cyclic word then the band
modules obtained from~$w$ and from its «rotations» and their
inverses are all isomorphic.

\bigskip

As usual, we say that an algebra~$\Lambda$ has \newterm{tame
representation type} if it is not of finite representation type and
for each $d\in\NN$ there exist finitely many
$(\kk[X],\Lambda)$-bimodules $M_1$,~\dots,~$M_t$ such that, up to
isomorphism, all but finitely many indecomposable $\Lambda$-modules of
dimension~$d$ are isomorphic to a $\Lambda$-module of the form
$\kk[X]/(X-\lambda)\otimes_{\kk[X]}M_i$ for some~$\lambda\in\kk$ and
some $i\in\{1,\dots,t\}$. When this is the case, we write $\mu(d)$ for
the minimal number of such bimodules needed for this. If there exists
an integer~$N$ such that $\mu(d)\leq N$ for all $d\in\NN$ then the
algebra~$\Lambda$ has \newterm{domestic representation type}.

The result of Gel\cprime fand and Ponomarev mentioned above implies
that string algebras are either of finite representation type or of
tame representation type, and it is easy to decide which among the
latter are domestic. If $\alpha$ is an arrow in~$Q$, we
denote~$\N(\alpha)$ the set of all cyclic words starting with~$\alpha$
and ending with an inverse letter --- which may well be empty, of
course. This set~$\N(\alpha)$ is closed under concatenation, so a
semigroup, and, in fact, easily seen to be a free one, freely
generated by its subset $\N(\alpha)\setminus\N(\alpha)^2$. It follows
from \cite{R:compact}*{Proposition 2} and observations made in the
introduction of~\cite{R:generic} that the following holds:

\begin{Proposition}
Let $(Q,I)$ be a string presentation. The following three conditions
are equivalent:
\begin{tfae}

\item The algebra $A\coloneqq\kk Q/I$ does not have domestic
representation type.

\item There are infinitely many primitive cyclic words in~$(Q,I)$.

\item There is an arrow~$\alpha$ in~$Q$ such that the
semigroup~$\N(\alpha)$ is neither empty nor cyclic.  \qed

\end{tfae}
\end{Proposition}

The other ingredient that we need in order to do what we want with
string algebras of non-domestic type is a nice result of H.J.\,Fine
and H.S.\,Wolf \cite{FW} on the combinatorics of words that we will
describe next; an excellent reference for this is the book~\cite{L} of
M.\,Lothaire, where the result we want appears as Proposition 1.3.5.
We fix a set~$\A$, which we call in this context the
\newterm{alphabet}, and write $\A^*$ for the free monoid generated
by~$\A$ ---in~\cite{L} the elements of~$\A^*$ are referred to
as~\emph{words}, but we will not do that since we are already using
that word (!) for something else. The \newterm{length} of an
element~$u$ of~$\A^*$ is what one expects, and we write it
as~$\abs{u}$. On the other hand, if~$u$,~$v$ and~$w$ are elements
of~$\A^*$, we say that $w$ is a common left factor of~$u$ and~$v$ if
there exist $u'$ and~$v'$ in~$\A^*$ such that~$u=wu'$ and $v=wv'$.

\begin{Proposition}\label{prop:words}\cite{L}*{Proposition 1.3.5}
Let $u$ and~$v$ be elements of~$\A^*$ of lengths~$n$ and~$m$, respectively,
and let $d\coloneqq\gcd(m,n)$. If two powers~$x^p$ and~$y^q$ of~$x$
and~$y$ have a common left factor of length at least equal to $n+m-d$,
then $x$ and~$y$ are powers of the same element of~$\A^*$. \qed
\end{Proposition}

We have now everything in place to do the construction that we need to
prove the following result:

\begin{Proposition}
Let $(Q,I)$ be a string presentation and let~$A\coloneqq\kk Q/I$.
If~the algebra $A$ is not domestic, then there are extensions between
indecomposable $A$-modules whose middle term has an arbitrarily large
number of direct factors.
\end{Proposition}

\begin{proof}
Let us suppose that the algebra~$A$ is not domestic. Lemma~3
in~\cite{R:generic} tells us that there are non-serial words~$x$,~$y$
and~$z$ such that $yxy$ and~$zxz$ are words, the first and last
letters of~$y$ are direct, and the first and last letters of~$z$ are
inverse. Let us notice first that
  \begin{equation}\label{eq:xyxz}
  \claim{$xy$ and $xz$ are not powers of the same word.}
  \end{equation}
Indeed, the last letter of $xy$ is direct, while that of $xz$ is
inverse.

Next, let us check that
  \begin{equation} \label{eq:xp}
  \claim[.7]{for all $n\geq3$ the words $(xyxz)^nxy$ and $xz(xyxz)^n$ are
  cyclic and primitive.}
  \end{equation}
Suppose, for example, that $n$ is a positive integer such that
$(xyxz)^nxy=w^t$ for some primitive cyclic word~$w$ and some
integer~$t$ with $t\geq2$. We claim that $xyxz$ is not a power of~$w$.
Indeed, if we had that $xyxz=w^s$, then it would follow that
$w^t=(xyxz)^nxy=w^{sn}xy$, so that $t\geq sn$ and $xy=w^{t-sn}$: but
then we would also have that $w^s=xyxz=w^{t-sn}xz$, so that $s\geq
t-sn$ and $xz=w^{(n+1)s-t}$, allowing us to conclude that $xz$ and
$xy$ are powers of~$w$, in~contradiction to~\eqref{eq:xyxz}. Our claim
thus holds. Now the words $(xyxz)^{n+1}$ and $w^t$ have $(xyxz)^n$ as
a common left factor (in the free monoid on the set of letters) and we
have shown that $xyxz$ and $w$ are not powers of a word: according to
Proposition~\ref{prop:words}, then, we have that $n\abs{xyzx} <
\abs{xyxz} + \abs{w}$, so that
  \[
  t(n-1)\abs{xyxz} < t\abs{w} = n\abs{xyxz} + \abs{xy}
        < (n+1)\abs{xyxz}
  \]
and $0 < 2-(t-1)(n-1)$. We thus see that $2>(t-1)(n-1)\geq n-1$ and,
therefore, that $n<3$. This proves~\eqref{eq:xp}.

Let now $p$ be a prime number different from the characteristic of our
ground field~$\kk$ and larger than~$7$, let $n$ be such that $p=2n+1$
and consider the two cyclic words 
  \[
  u = (xyxz)^nxy, 
  \qquad
  v = xz(xyxz)^n
  \]
which are primitive since $n\geq3$. Let $\gamma$ and~$\beta$ be the
first and last letters of~$y$, and $\alpha^{-1}$ and~$\delta^{-1}$ the
first and last letters of~$z$. Clearly $zxy$ is a factor of~$u$, so
that so is $\delta^{-1}x\gamma$. In fact, it has many such factors but
we choose one arbitrarily. Similarly, $yxz$ is a factor of~$v$, and
therefore so is $\beta x\alpha^{-1}$, and we fix an appearance of this
factor in~$v$. Since $x$ has positive length, the special biseriality
of~$(Q,I)$ implies at once that $\beta\delta\in I$ and
$\alpha\gamma\in I$.
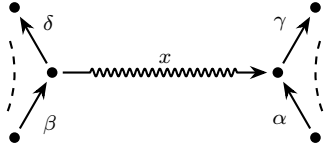
  \[
  \begin{tikzpicture}[font=\footnotesize]
  \fill (0,0)  circle(2pt) coordinate (1)
        +(+120:1) circle(2pt) coordinate (1u)
        +(-120:1) circle(2pt) coordinate (1d)
        ++(3,0)  circle(2pt) coordinate (2)
        +(+60:1) circle(2pt) coordinate (2u)
        +(-60:1) circle(2pt) coordinate (2d)
        ;
  \draw[snake arrow] (1) -- node[above] {$x$} (2);
  \draw[quiver arrow] (1) -- node[anchor=south west] {$\delta$} (1u);
  \draw[quiver arrow] (1d) -- node[anchor=north west] {$\beta$} (1);
  \draw[quiver arrow] (2) -- node[anchor=south east] {$\gamma$} (2u);
  \draw[quiver arrow] (2d) -- node[anchor=north east] {$\alpha$} (2);
  \draw[dashed,thick] ($(2)+(2,0)+(200:1.5)$) arc (200:160:1.5);
  \draw[dashed, thick] ($(1)-(2,0)+(-20:1.5)$) arc (-20:20:1.5);
  \end{tikzpicture}
  \]

As $u$ and~$v$ are primitive cyclic words or, equivalently, bands, we
can construct from them the band modules $B(u)\coloneqq B(u,1,1)$ and
$B(v)\coloneqq B(v,1,1)$ with parameters~$1\in\kk$ and $1\in\NN$.
These are irreducible modules of dimensions equal to the lengths
of~$u$ and~$v$, respectively. The module~$B(u)$ has a basis
corresponding to the vertices the cyclic word~$u$ visits along its way
in~$Q$ and the same is true for~$B(v)$. We name~$i_1$ and~$i_2$ the
basis elements of~$B(v)$ which correspond to the source and the target
of the arrow~$\beta$ that appears in the factor~$\beta x\alpha^{-1}$
that we have chosen in~$v$. Similarly, we let $j_1$ and~$j_2$ be the
basis vectors of~$B(u)$ corresponding to the source and the target of
the arrow~$\delta$ that appears in the factor~$\delta^{-1}x\gamma$
that we have chosen in~$u$. The direct sum $M\coloneqq B(u)\oplus
B(v)$ can be schematically described by the following drawing ---in
which we ignore for now the dashed arrows.
  
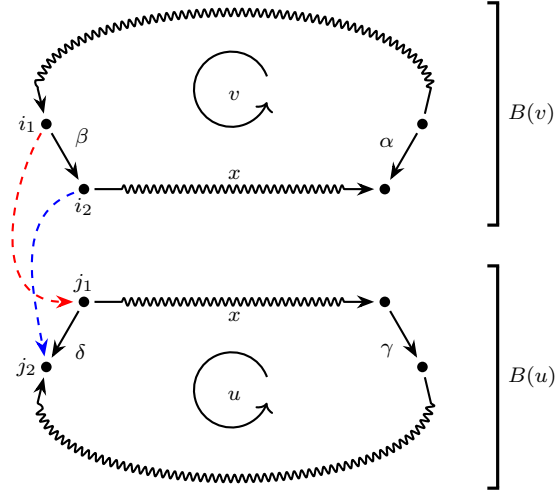
\[
  \begin{tikzpicture}[font=\footnotesize]
  \fill (0,1.5)  circle(2pt) coordinate (1v)
                node[anchor=north] {$i_2$}
        ++(4,0)  circle(2pt) coordinate (2v)
        +(+60:1) circle(2pt) coordinate (3v)
        (1v) + (180-60:1) circle(2pt) coordinate (4v)
                node[anchor=east] {$i_1$}
        ;
  \fill (0,0)  circle(2pt) coordinate (1u)
                node[anchor=south] {$j_1$}
        ++(4,0)  circle(2pt) coordinate (2u)
        +(-60:1) circle(2pt) coordinate (3u)
        (1u) + (180+60:1) circle(2pt) coordinate (4u)
                node[anchor=east] {$j_2$}
        ;
  \draw[quiver arrow] (4v) -- node[anchor=south west] {$\beta$} (1v);
  \draw[snake arrow] (1v) -- node[above] {$x$} (2v);
  \draw[quiver arrow] (3v) -- node[anchor=south east] {$\alpha$} (2v);
  \draw[snake arrow] (3v) .. controls +(+60:2.25) and +(180-60:2.25) .. (4v);
  \draw[->, thick] (2,2.75)+(30:0.5)  arc[radius=0.5, start angle=20, delta angle=320];
  \node[anchor=center] at (2,2.75) {$v$};
  \draw[quiver arrow] (1u) -- node[anchor=north west] {$\delta$} (4u);
  \draw[snake arrow] (1u) -- node[below] {$x$} (2u);
  \draw[quiver arrow] (2u) -- node[anchor=north east] {$\gamma$} (3u);
  \draw[snake arrow] (3u) .. controls +(-60:2.25) and +(180+60:2.25) .. (4u);
  \draw[->, thick] (2,-1.25)+(30:0.5)  arc[radius=0.5, start angle=20, delta angle=320];
  \node[anchor=center] at (2,-1.25) {$u$};
  \draw[quiver arrow, red, dashed] 
        (4v) .. controls +(-120:1) and +(+180:1.25) .. (1u);
  \draw[quiver arrow, blue, dashed] 
        (1v) .. controls +(200:1) and +(+100:1) .. (4u);
  \draw[{Bar[right]}-{Bar[left]}, very thick]  
        (5.5,1) -- node[right] {$B(v)$} +(0,3);
  \draw[{Bar[right]}-{Bar[left]}, very thick]  (5.5,-2.5) 
        -- node[right] {$B(u)$} +(0,3);
  \end{tikzpicture}
  \]
We now construct a $\kk Q$-module~$M'$. As a vector space, $M'$
coincides with~$M$. To give the $\kk Q$-module structure on~$M'$ it is
enough to describe how the vertices and the arrows of~$Q$ act on it,
and this is what we do.
\begin{itemize}

\item We let the vertices of~$Q$ act on~$M'$ as they act on~$M$.

\item If $\eta$ is an arrow in~$Q$ and $k$ is a basis vector of~$M'$
such that the pair $(\eta,k)$ is neither $(\beta,i_1)$
nor~$(\delta,i_2)$, we let $\eta$ act on~$k$ as it acts in~$M$.

\item Finally, we put $i_1\cdot\beta = i_2+j_1$ and $i_2\cdot\delta =
-j_2$.

\end{itemize}
If $\rho$ is a path in~$Q$ which belongs to~$I$, then the
straightforward consideration of the few possibilities that exist
shows that $\rho$ acts as zero on~$M'$: this means that $M'$ is an
$A$-module. Moreover, it is clear that we have an extension
  \[
  \begin{tikzcd}
  0 \arrow[r]
    & B(u) \arrow[r]
    & M' \arrow[r]
    & B(v) \arrow[r]
    & 0
  \end{tikzcd}
  \]
Now it is easy to see that the module $M'$ is isomorphic to the module
that can be constructed from the cyclic word~$uv$ just as band modules
are constructed using the parameters $\lambda=-1$ and $n=1$. It is
not, though, a band module: indeed, by construction we have that
$uv=(xyxz)^{2n+1}=(xyxz)^p$, so that $uv$ is \emph{not} a primitive
cyclic word. As it has a rotational symmetry of order~$p$ and $p$~is
coprime with the characteristic of our algebraically closed
field~$\kk$, the module $M'$ is a direct sum of at least~$p$
indecomposable direct summands: this proves the proposition.
\end{proof}


\end{document}